\documentclass[12pt,a4paper,final]{amsart}
\usepackage{version}
\usepackage[notcite,notref]{showkeys}
\usepackage{amssymb}
\usepackage{graphicx}

\input epsf

\newtheorem{theorem}{Theorem}[section]
\newtheorem{lemma}[theorem]{Lemma}

\newtheorem{proposition}[theorem]{Proposition}

\theoremstyle{definition}
\newtheorem{definition}[theorem]{Definition}

\theoremstyle{remark}
\newtheorem{remark}[theorem]{Remark}

\numberwithin{equation}{section}

\DeclareMathOperator{\Id}{Id}
\DeclareMathOperator{\dimH}{dim_H}

\newcommand{\spt}{\text{spt}\, }

\begin{document}

\title[Non-degenerate families of projections]
{Hausdorff dimension and non-degenerate families of projections}

\author[E. J\"arvenp\"a\"a]{Esa J\"arvenp\"a\"a$^1$}
\address{Department of Mathematical Sciences,  P.O. Box 3000,
         90014 University of Oulu, Finland$^{1,2}$}
\email{esa.jarvenpaa@oulu.fi$^1$}

\author[M. J\"arvenp\"a\"a]{Maarit J\"arvenp\"a\"a$^2$}
\email{maarit.jarvenpaa@oulu.fi$^2$}

\author[T. Keleti]{Tam\'as Keleti$^3$}
\address{Department of Analysis, E\"otv\"os Lor\'and University, P\'azm\'any 
        P\'eter s\'et\'any 1/C, H-1117 Budapest, Hungary$^3$}
\email{elek@cs.elte.hu$^3$}

\thanks{We acknowledge the support of the Centre of Excellence in Analysis and 
Dynamics Research funded by the Academy of Finland. The third author was also
supported by OTKA grant no. 72655 and J\'anos Bolyai Fellowship.}

\subjclass[2000]{28A80, 37C45}
\keywords{Projection, Hausdorff dimension, measure}

\begin{abstract} 
We study parametrized families of orthogonal projections for which the 
dimension of the parameter space is strictly less than that of the Grassmann 
manifold. We answer the natural question of how much the Hausdorff 
dimension may decrease by verifying the best possible lower
bound for the dimension of almost all projections of a finite measure. We also 
show that a similar result is valid for smooth families of maps from
$n$-dimensional Euclidean space to $m$-dimensional one.  
\end{abstract}
\maketitle

\section{Introduction}\label{intro}

The behaviour of different concepts of dimensions of sets and measures under 
projections
has been investigated intensively for several decades. The study was initiated
by Marstrand \cite{Mar} in the 1950's. Mattila \cite{Mat1} considered 
Hausdorff dimension of sets in the 1970's, and in the late 1980's and
in the 1990's several authors contributed to the field. In 2000 Peres and
Schlag \cite{PS} proved a very general result concerning transversal families
of mappings and Sobolev dimension. For a more detailed account of the history,
see the survey of Mattila \cite{Mat3}. 

All the above results concerning Hausdorff dimension may be simplified 
by stating that the dimension is preserved under almost all projections. The 
essential assumption is transversality which is guaranteed in many cases by 
identifying the parameter space with an 
open subset of the Grassmann manifold. The question we are addressing is that
how much the dimension may drop under almost all projections provided that
the dimension of the parameter space is less than that of the Grassmann
manifold. The following conclusion can be drawn from \cite{PS}:
Fubini's theorem implies that for a given set or
a measure the dimension is preserved for almost all projections in almost all
$k$-dimensional families for any $k$. Hence, for a given measure
one obtains information for typical families. However, in general
there is no way to conclude whether a given family is typical for a given 
measure. Furthermore, the results of \cite{PS} concerning exceptional sets
of parameters may be applied if $k$ is large enough but the bounds obtained in 
this way are not optimal except in a few special cases (see Remark 
\ref{PSremark}).

The aforementioned question appears naturally in applications. For 
example, the study of projections of measures invariant under the geodesic
flow on $n$-dimensional Riemann manifolds leads to a 1-dimensional family
of projections from a $2(n-1)$-dimensional space onto an $(n-1)$-dimensional
space (see \cite{LL,JJL}). 
Another interesting example
is Falconer's \cite{F2} attempt to prove that there are no Besicovitch 
$(n,m)$-sets for $m\ge 2$. A set $A\subset\mathbb R^n$ is a Besicovitch 
$(n,m)$-set if the $n$-dimensional Lebesgue measure of $A$ is zero
and $A$ contains a translate of every $m$-dimensional 
linear subspace of $\mathbb R^n$. There is a gap in the proof related 
to the issue of the behaviour of the dimension under a $k$-dimensional family
of projections onto $m$-planes where $k$ is less than the dimension of the 
Grassmann manifold $G(n,m)$. 

To obtain results for almost all projections it is not sufficient to assume 
that the projection family is smooth since it is possible to parametrize
exceptional projections with many parameters. To prevent this from happening,
we assume that the family is locally embeddable into the Grassmann manifold 
guaranteeing that the mapping is changed when the parameter is changed. 

The cases of 1-dimensional families of projections onto $m$-planes and 
general families of projections onto lines or hyper-planes are dealt in
\cite{JJLL}. In this paper we solve completely the general case by 
proving the best possible almost sure lower bound in a 
$k$-dimensional family of projections onto $m$-planes in $\mathbb R^n$ (see 
Theorem~\ref{projresult}). We also verify that the corresponding 
result is valid for 
parametrized families of smooth maps between $\mathbb R^n$ and $\mathbb R^m$
(see Theorem~\ref{generalresult}). 

When applying our result to the setting of \cite{F2} we observe that the
dimension of the parameter space is too small to obtain the desired result
except in the case of Besicovitch $(n,n-1)$-sets. Since our result is the best 
possible one for general families, this means that if
there is a way to fix the gap in \cite{F2} for Besicovitch $(n,m)$-sets with 
$m<n-1$, one needs to utilize the special properties
of the projection family constructed in \cite{F2}.
   
The paper is organized as follows. In Section~\ref{trans} we give the basic
definitions and the auxiliary results needed later. Our main theorem 
concerning families of projections is verified in Section 
\ref{projfamilies} and generalized to families of 
smooth maps in Section~\ref{generalfamilies}.

\section{Basic definitions}\label{trans}

In this section we introduce the notation used throughout this paper.
Let $m$ and $n$ be integers with $0<m<n$ and let $\mu$ be a finite 
Radon measure on $\mathbb R^n$ with compact support.
The Hausdorff dimension $\dim$ of $\mu$ is defined in terms of local 
dimensions as follows:
\begin{equation}\label{locdimdef}
\dim\mu=\sup\{s\ge0\mid\liminf_{r\to0}\frac{\log\mu(B(x,r))}{\log r}\ge s
        \text{ for }\mu\text{-almost all }x\in\mathbb R^n\},
\end{equation}
where $B(x,r)$ is the open ball with centre at $x$ and radius $r$.
Equivalently,
\begin{equation}\label{setdef}
\dim\mu=\inf\{\dim A\mid A\subset\mathbb R^n\text{ is a Borel set with }
\mu(A)>0\}.
\end{equation}
For this equivalence and other properties of dimensions of measures see
\cite[Proposition 10.2]{F3}. It follows easily from \eqref{locdimdef} that
\begin{equation}\label{energydim}
I_t(\mu)<\infty\implies\dim\mu\ge t,
\end{equation}
where
\[
I_t(\mu)=\iint\vert x-y\vert^{-t}\,d\mu(x)\,d\mu(y)
\]
is the $t$-energy of $\mu$.

Let $k$ be an integer with $0<k<m(n-m)$. Note that $m(n-m)$ is the dimension
of the Grassmann manifold $G(n,m)$ of all $m$-dimensional linear subspaces of 
$\mathbb R^n$. Supposing that $\Lambda\subset\mathbb R^k$ is open, we consider 
parametrized families 
$\{F_\lambda:\mathbb R^n\to\mathbb R^m\mid\lambda\in\Lambda\}$ of smooth maps.
We denote the orthogonal projection in $\mathbb R^n$ onto an $m$-dimensional 
subspace $V\in G(n,m)$ by $\Pi_V$.
When investigating parametrized families of orthogonal projections 
$\Pi_{V_\lambda}:\mathbb R^n\to V_\lambda$ onto $V_\lambda\in G(n,m)$,
we consider them as
mappings from $\mathbb R^n$ to $\mathbb R^n$. Clearly, such a family 
could also be viewed as a family from $\mathbb R^n$ to $\mathbb R^m$
by identifying the range $V_\lambda\in G(n,m)$ with $\mathbb R^m$ in a
systematic manner, for example, by fixing an orthonormal basis of some 
$V_{\lambda^0}$ and by rotating the basis to $V_\lambda$ by a rotation which
rotates $V_\lambda$ to $V_{\lambda^0}$. Since the identification is neither 
unique nor essential, we omit it.

The image of
a measure $\mu$ under a map $T:X\to Y$ is denoted by $T_*\mu$, that is, 
$T_*\mu(A)=\mu(T^{-1}(A))$ for all $A\subset Y$. If $\mu$ is a Radon measure 
on $X$ with compact support and $T$ is a Lipschitz map, the image measure 
$T_*\mu$ is a Radon measure on $Y$ with compact support
\cite[Theorem 1.18]{Mat2}. We use the notation $\spt\mu$ for the
support of a measure $\mu$. 
Obviously, 
\begin{equation}\label{naturalbounds}
\dim\mu-(n-m)\le\dim (\Pi_V)_*\mu\le\min\{\dim\mu,m\}
\end{equation}
for all $V\in G(n,m)$. 

For $r=2,\dots,n$ a simple $r$-vector in $\mathbb R^n$ is denoted by
$v_1\wedge\dots\wedge v_r$ where $v_i\in\mathbb R^n$ for $i=1,\dots,r$.
A non-zero simple $r$-vector $v_1\wedge\dots\wedge v_r$ determines uniquely
an $r$-plane $\langle v_1,\dots,v_r\rangle\in G(n,r)$ (see 
\cite[Section 1.6]{Fe}). The norm of a simple $r$-vector is given by
\[
\Vert v_1\wedge\dots\wedge v_r\Vert=\sqrt{\det(DD^T)}
\]
where $D$ is the $r\times n$-matrix whose $i^{th}$ row consists of the 
coordinates of $v_i$. Note that the norm is equal to the $r$-dimensional volume
of the parallelepiped spanned by $v_1,\dots,v_r$. In particular, if the vectors 
$v_1,\dots,v_l\in\mathbb R^n$ are perpendicular to the vectors 
$u_1,\dots,u_t\in\mathbb R^n$, we have
\begin{equation}\label{perpnorm}
\Vert v_1\wedge\dots\wedge v_l\wedge u_1\wedge\dots\wedge u_t\Vert=
  \Vert v_1\wedge\dots\wedge v_l\Vert\cdot\Vert u_1\wedge\dots\wedge u_t\Vert.
\end{equation}
A linear map 
$L:\mathbb R^n\to\mathbb R^m$ can be
naturally extended to a linear map 
$\wedge_r L:\Lambda_r\mathbb R^n\to\Lambda_r\mathbb R^m$ between the vector 
spaces of $r$-vectors. The norm of $\wedge_r L$ is defined by
\[
\Vert\wedge_r L\Vert=\sup\{\Vert\wedge_r L(\xi)\Vert\mid \xi\text{ is a 
  simple }r\text{-vector with }\Vert\xi\Vert=1\}.
\]
Note that $\Vert\wedge_n L\Vert=|\det L|$.

The following well-known lemma plays a fundamental role in our approach.
We use
the notation $\mathcal L^k$ for the Lebesgue measure on $\mathbb R^k$. 
In the case $k=1$ the Lebesgue measure is denoted by $\mathcal L$.

\begin{lemma}\label{lemma1} Let $n,m,k$ and $l$ be integers satisfying
$0<k<m(n-m)$ and $l\ge m$. Let $\Lambda\subset\mathbb R^k$ be bounded and let
$\{F_\lambda:\mathbb R^n\to\mathbb R^l\mid\lambda\in\Lambda\}$ be a 
parametrized family of Lipschitz maps such that for all $\lambda\in\Lambda$
there exists a smooth $m$-dimensional submanifold of $\mathbb R^l$ containing 
$F_\lambda(\mathbb R^n)$. Assume that 
$\mu$ is a finite Radon measure on $\mathbb R^n$ with compact support
and $r$ is a positive real number such that $r\le m$. 
Suppose that for all $z\in\spt\mu$ there exist $\varepsilon>0$ and $C>0$ such 
that for all $x\ne y\in B(z,\varepsilon)$ and for all $\delta>0$
\begin{equation}\label{smallmeasure}
\mathcal L^k(\{\lambda\in\Lambda\mid\vert F_\lambda(x)-F_\lambda(y)\vert
  \le\delta\})\le C\delta^r\vert x-y\vert^{-r}.
\end{equation}
Then $\dim(F_\lambda)_*\mu=\dim\mu$ for $\mathcal L^k$-almost
all $\lambda\in\Lambda$ provided that $\dim\mu\le r$. Furthermore, 
$\dim(F_\lambda)_*\mu\ge r$ for $\mathcal L^k$-almost all $\lambda\in\Lambda$
provided that $\dim\mu>r$.
Finally, for $\mathcal L^k$-almost all $\lambda\in\Lambda$ the image measure
$(F_\lambda)_*\mu$ is absolutely continuous with respect to the $m$-dimensional 
Hausdorff measure $\mathcal H^m$ if $\dim\mu>m$ and $r=m$.
\end{lemma}

\begin{proof} Covering the compact set $\spt\mu$ by a finite collection
of open balls $B(z_i,\varepsilon_i)$ and letting 
$\mu_i=\mu|_{B(z_i,\varepsilon_i)}$
be the restriction of $\mu$ to the ball $B(z_i,\varepsilon_i)$, we have
$\dim\mu=\min_i\dim\mu_i$ and 
$\dim(F_\lambda)_*\mu=\min_i\dim(F_\lambda)_*\mu_i$. Therefore, we may restrict
our consideration to a restricted measure $\mu_i$. 

The first two claims follow similarly as 
in \cite[Lemmas 2.1 and 2.2]{JJLL}. Even though \cite{JJLL} deals with 
projections the only essential assumption is that $(F_\lambda)_*\mu$ is a 
Radon measure.

For the last claim proceed as in the proof of \cite[Lemma 2.2]{JJLL} to find
a restriction of $\mu$ having finite $m$-energy and apply the proof of 
\cite[Theorem 9.7]{Mat2}. Here we use the assumption
that the range of $F_\lambda$ is contained in a smooth $m$-dimensional 
submanifold $M_\lambda$ which implies that $\mathcal H^m|_{M_\lambda}(B(x,r))$
is comparable to $r^m$. 
\end{proof}

\begin{remark}\label{localcoor} 
In the proof of Theorem~\ref{projresult} we need local coordinates on $G(n,m)$
and the following choice turns out to be useful.
Consider $V\in G(n,m)$. Let $\{e_1,\dots,e_m\}$ and 
$\{e_{m+1},\dots,e_n\}$ be orthonormal bases of $V$ and 
its orthogonal complement $V^\perp\in G(n,n-m)$, respectively. 
One may choose local coordinates on $G(n,m)$ near $V$ in terms of 
rotations of the basis vectors 
$\{e_1,\dots,e_m\}$ in the following manner: For $i=1,\dots,m$ and 
$j=m+1,\dots,n$, let $-\tfrac\pi 4<\alpha_{ij}<\tfrac\pi 4$ be the components 
of $\alpha\in ]-\tfrac\pi 4,\tfrac\pi 4[^{m(n-m)}$.
Rotating $e_i$ by the angle $\alpha_{ij}$ towards $e_j$ for all $i$ and $j$ 
gives local coordinates for the $m$-plane $V(\alpha)$ spanned by the rotated 
vectors. More precisely, 
$V(\alpha)=\langle e_1(\alpha),\dots,e_m(\alpha)\rangle$, 
where $e_i(\alpha)=\prod_{j=m+1}^nR^{ij}(\alpha_{ij})e_i$ is an ordered 
product for all $i=1,\dots,m$ and
\[
\bigl(R^{ij}(\beta)x\bigr)_l=
\begin{cases} x_i\cos\beta - x_j\sin\beta, &\text{if }l=i\\
              x_i\sin\beta + x_j\cos\beta, &\text{if }l=j\\ 
              x_l, &\text{otherwise.}
\end{cases}
\]
For the proof of the fact that these rotations give local coordinates, see 
\cite[Remark 2.4]{JJLL}. Further, let
$\{\frac\partial{\partial\alpha_{ij}}\mid i=1,\dots,m,j=m+1,\dots,n\}$ be 
the basis of the tangent space $T_VG(n,m)$ obtained in this way. 
A straightforward calculation shows that for any $z\in V^\perp$, $w\in V$,
$i\in\{1,\dots,m\}$ and $j\in\{m+1,\dots,n\}$ we have
\begin{equation}\label{linearmap}
\left.\frac{\partial\Pi_{V(\alpha)}(z)}{\partial\alpha_{ij}}\right|_{\alpha=0}
  =z_je_i\quad\text{and}\quad
\left.\frac{\partial\Pi_{V(\alpha)}(w)}{\partial\alpha_{ij}}\right|_{\alpha=0}
  =w_ie_j.
\end{equation}

\end{remark}

\section{Families of projections}\label{projfamilies}

In this section we state and prove our main theorem concerning parametrized
families of orthogonal projections. 
We equip the Grassmann manifold $G(n,m)$ with a Riemann metric and 
define the class of families of projections we are working with.

\begin{definition}\label{nondegenerate}
Let $\Lambda\subset\mathbb R^k$ be open. A parametrized family 
$\{\Pi_{V_\lambda}\mid\lambda\in\Lambda\}$ of orthogonal projections in 
$\mathbb R^n$ onto $m$-planes is called \emph{non-degenerate} if the mapping
$\lambda\mapsto V_\lambda$ is continuously differentiable and the derivative
$D_\lambda V_\lambda$ is injective for all $\lambda\in\Lambda$.
\end{definition}

For all $x\in\mathbb R$, we denote by $]x]$ 
the smallest integer $q\ge 0$ such that $x\le q$. Furthermore, given an 
integer $0<k<m(n-m)$, define
\begin{equation}\label{p}
p(l)=n-m-\left]\frac{k-l(n-m)}{m-l}\right]
\end{equation}
for $l=0,\dots,m-1$.

\begin{theorem}\label{projresult}
Let $\Lambda\subset\mathbb R^k$ be an open set and let $\mu$ be a finite 
Radon measure on $\mathbb R^n$ with compact support. Assume that a family 
$\{\Pi_{V_\lambda}\mid\lambda\in\Lambda\}$ of orthogonal projections in 
$\mathbb R^n$ onto $m$-planes is non-degenerate. Then for all $l=0,\dots,m-1$ 
and for $\mathcal L^k$-almost all $\lambda\in\Lambda$ 
\begin{equation}\label{lowerlimit}
\dim (\Pi_{V_\lambda})_*\mu\ge
\begin{cases} \dim\mu-p(l), &\text{if }p(l)+l\le\dim\mu\le p(l)+l+1,\\
               l+1, &\text{if }p(l)+l+1\le\dim\mu\le p(l+1)+l+1.
\end{cases}
\end{equation}
Furthermore, for $\mathcal L^k$-almost all $\lambda\in\Lambda$ the projected
measure $(\Pi_{V_\lambda})_*\mu$ is absolutely 
continuous with respect to $\mathcal H^m$ provided that
$\dim\mu>p(m-1)+m$. 
The lower bounds given in \eqref{lowerlimit} and the condition for the 
absolute continuity are the best possible ones. 
\end{theorem}  

\begin{remark}\label{full}
a) Theorem~\ref{projresult} is valid if $D_\lambda V_\lambda$ is injective only
for $\mathcal L^k$-almost all $\lambda\in\Lambda$ since, by continuity, the set
$N=\{\lambda\in\Lambda\mid D_\lambda V_\lambda\text{ is not injective}\}$
is closed, and therefore, one may 
replace $\Lambda$ by $\Lambda\setminus N$.

b) The injectivity assumption is natural: Theorem~\ref{projresult} is not 
necessarily true without it. Indeed, by the sharpness of 
\eqref{lowerlimit}, there is 
a $(k-1)$-dimensional family for which the lower bound in \eqref{lowerlimit}
is obtained. We extend the family to a $k$-dimensional one by
adding an extra parameter which does not change the maps. The extension does not
affect the dimensions of the projections, and it follows from \eqref{p} that
\eqref{lowerlimit} is not valid for the extended family for 
which the injectivity fails.  

c) The fact that the function $p$ in \eqref{p} is increasing can be 
gleaned from Figure~\ref{pic5}. Indeed, after filling the $l$ lowest rows in
Figure~\ref{pic5} with dots, one is left with $k-l(n-m)$ dots, where $k$ is 
the original number of dots. Proceed by filling the columns from left. The 
number of the columns needed is $]\tfrac{k-l(n-m)}{m-l}]$ implying that
$p(l)$ is the number of the remaining unoccupied 
columns. When increasing $l$ by one, one needs to move dots from
the last occupied column to the unoccupied slots on the $(l+1)^{\text{th}}$ row.
This means that the number of unoccupied columns may increase but not decrease. 
\end{remark}

\begin{figure}[htp!]
\begin{center}
{\includegraphics[scale=0.6]{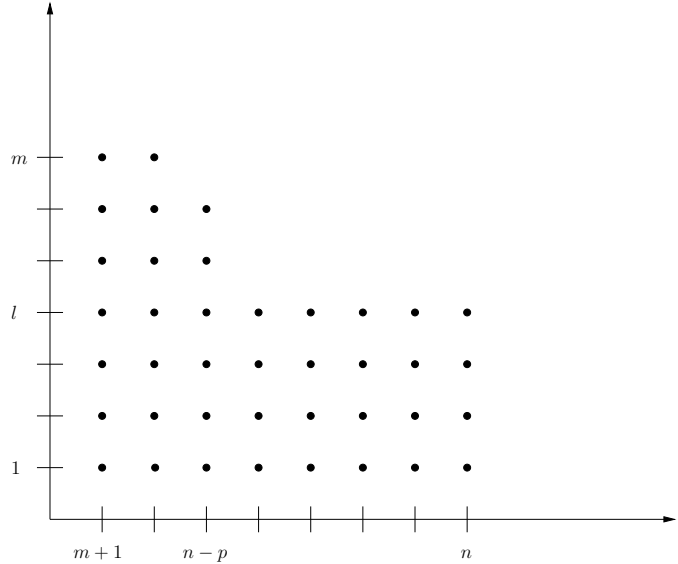}}
\caption{The dotting that explains the definition of $p(l)$ and the sharpness 
of the lower bounds in Theorem~\ref{projresult}.}\label{pic5}
\end{center}
\end{figure}

We continue by proving a technical lemma. 

\begin{lemma}\label{independent}
Let $A_1,\dots,A_k:\mathbb R^{n-m}\to\mathbb R^m$ be linear maps and let 
$C,d>0$. Assume that $\Vert A_i\Vert<C$ for all $i=1,\dots,k$ and 
$\Vert A_1\wedge\dots\wedge A_k\Vert>d$ where $A_1,\dots,A_k$ are considered as
vectors in $\mathbb R^{m(n-m)}$.
Suppose that for some integers $1\le t\le n-m$ and $0\le l\le m-1$ 
we have $k>m(t-1)+l(n-m-t+1)$. Then there exist $d'$ depending only on $C$, 
$d$ and $n$, and a $t$-dimensional subspace $W\subset\mathbb R^{n-m}$ such that
for all $z\in W\setminus\{0\}$ there are $j_1,\dots,j_{l+1}$ satisfying 
\begin{equation}\label{bigvol2}
\Vert A_{j_1}(z)\wedge\dots\wedge A_{j_{l+1}}(z)\Vert>d'|z|^{l+1}.
\end{equation}
In particular, $\dim\langle A_1(z),\dots,A_k(z)\rangle\ge l+1$ for all 
$z\in W\setminus\{0\}$.
\end{lemma}

\begin{proof}
The heuristic idea behind the proof is as follows: assuming that the last
claim is not true and using the fact that a linear map is uniquely determined
by the images of the basis vectors, one can find $n-m-t+1$ orthonormal vectors 
having at most $l$ linearly independent images. 
Since the remaining $t-1$ basis vectors have at most $m$ linearly 
independent images, there are at most $m(t-1)+l(n-m-t+1)<k$ independent maps
in the family $\{A_1,\dots,A_k\}$ which is a contradiction. Hence, the last 
claim holds which,
in turn, implies \eqref{bigvol2} since the left hand side of \eqref{bigvol2}
is continuous and homogeneous of degree $l+1$.

To make the above idea rigorous, assume that for some $\tilde d>0$ and 
for all $t$-dimensional subspaces $W\subset\mathbb R^{n-m}$ there is 
$z\in W\setminus\{0\}$ such that the inequality  
$\Vert A_{j_1}(z)\wedge\dots\wedge A_{j_{l+1}}(z)\Vert\le\tilde d|z|^{l+1}$
holds for all $j_1,\dots,j_{l+1}$. By homogeneity, one can find inductively 
orthonormal vectors $z_1,\dots,z_{n-m-t+1}$ with
\begin{equation}\label{smallvol2}
\Vert A_{j_1}(z_i)\wedge\dots\wedge A_{j_{l+1}}(z_i)\Vert\le\tilde d
\end{equation}
for all $i=1,\dots,n-m-t+1$ and for all $j_1,\dots,j_{l+1}$. 
We will prove that \eqref{smallvol2} implies $\tilde d>\frac d{C'}$ where $C'$ 
depends only on $C$ and $n$. This means that \eqref{bigvol2} holds with
$d'=\frac d{C'}$ since otherwise by taking $\tilde d=d'$ we would get a 
contradiction.

Extend $\{z_1,\dots,z_{n-m-t+1}\}$ to an orthonormal basis 
$\{z_1,\dots,z_{n-m}\}$ of $\mathbb R^{n-m}$, fix an orthonormal basis in 
$\mathbb R^m$, and view $A_j$ as an $m\times(n-m)$-matrix determined by these 
bases. Let $D$ be the $k\times m(n-m)$-matrix whose $j^{th}$ row consists of 
the elements of $A_j$. According to the Cauchy-Binet formula 
\[
\Vert A_1\wedge\dots\wedge A_k\Vert=\sqrt{\det(DD^T)}=\sqrt{\sum_{d_k}d_k^2}
\]
where the sum is over all $k\times k$-minors $d_k$ of $D$. Since 
$k>m(t-1)+l(n-m-t+1)$, the pigeonhole principle implies
that any $k\times k$-submatrix of $D$ contains at least $l+1$ columns picked 
out from the set of columns determined by $A_1(z_i)$ for some 
$i=1,\dots,n-m-t+1$. 
Applying the Cauchy-Binet formula in \eqref{smallvol2} gives that any
$(l+1)\times(l+1)$-minor picked out from these $l+1$ columns has absolute value
at most $\tilde d$, and therefore, every term in the expression of any 
minor $d_k$ 
contains a factor at most $\tilde d$. From the fact that 
$\Vert A_i\Vert\le C$ for all $i=1,\dots,k$,
we derive
\[
d<\Vert A_1\wedge\dots\wedge A_k\Vert\le C'\tilde d
\]
where $C'$ depends on $C$, $k$, $l$, $n$ and $m$. Since 
$l<m<n$ and $k\le m(n-m)$, we may choose $C'$ in such a way that it depends 
only on $C$ and $n$. Therefore, as we claimed $\tilde d>\frac d{C'}$, which
completes the proof of \eqref{bigvol2}.

Finally, the last 
claim follows since $\dim\langle A_{j_1}(z),\dots,A_{j_{l+1}}(z)\rangle=l+1$.
\end{proof}

\begin{remark}\label{vgenerates}
Define $F:G(n,m)\times\mathbb R^n\to\mathbb R^n$ by $F(V,z)=\Pi_V(z)$. Given
$V_0\in G(n,m)$, by Remark~\ref{localcoor} (see \eqref{linearmap}),  
the formula
$\frac{\partial F(V_0,\cdot)}{\partial\alpha_{ij}}$ 
defines a linear map from $V_0^\perp$ to $V_0$, where $\alpha_{ij}$ are the
local coordinates around $V_0$ defined in Remark~\ref{localcoor}. 
The maps $\frac{\partial F(V_0,\cdot)}{\partial\alpha_{ij}}$ are clearly 
linearly independent for $i=1,\dots,m$ and 
$j=m+1,\dots,n$, and since 
$\dim G(n,m)=m(n-m)$, all linear maps from $V_0^\perp$ to $V_0$
are linear combinations of them. In particular, the map defined by
$D_VF(V_0,\cdot)(v)$ for all $v\in T_{V_0}G(n,m)$ is a bijection from
$T_{V_0}G(n,m)$ onto the space of linear maps from $V_0^\perp$ to $V_0$. 
Linear maps from $V_0$ to $V_0^\perp$ can be characterized similarly.

Let $\lambda^0\in\Lambda\subset\mathbb R^k$, $z_1\in V_{\lambda^0}$ and
$z_2\in V_{\lambda^0}^\perp$. By the above observation
\begin{equation}\label{images}
D_\lambda\Pi_{V_{\lambda^0}}(z_1)(w)\in V_{\lambda^0}^\perp\text{ and }
D_\lambda\Pi_{V_{\lambda^0}}(z_2)(w)\in V_{\lambda^0}
\end{equation}
for all $w\in\mathbb R^k$,
where $D_\lambda\Pi_{V_{\lambda^0}}(z)$ is the derivative of the map
$\lambda\mapsto\Pi_{V_\lambda}(z)$ at $\lambda^0$. Since 
$D_\lambda\Pi_{V_{\lambda^0}}(z_1+z_2)(w)=D_\lambda\Pi_{V_{\lambda^0}}(z_1)(w)
  +D_\lambda\Pi_{V_{\lambda^0}}(z_2)(w)$,
condition \eqref{images} implies that 
\[
\Pi_{V_{\lambda^0}}\bigl(\wedge_r D_\lambda\Pi_{V_{\lambda^0}}(z_1+z_2)(\xi)\bigr)
  =\wedge_r D_\lambda\Pi_{V_{\lambda^0}}(z_2)(\xi)
\]
for any $r\le m$ and for any simple $r$-vector $\xi$ on $\mathbb R^k$. 
Therefore, we have for any $r\le m$ that 
\begin{equation}\label{projbound}
\Vert\wedge_r D_\lambda\Pi_{V_{\lambda^0}}(z_1+z_2)\Vert
  \ge\Vert\wedge_r D_\lambda\Pi_{V_{\lambda^0}}(z_2)\Vert.
\end{equation}
\end{remark}
  
We will be working with families for which the transversality condition 
\cite[Definition 7.2]{PS} is not valid. The following proposition, which may be 
regarded as a partial transversality condition, is our main tool. 
In Proposition \ref{dotbound} the projection family does not need to be 
non-degenerate. Since our main interest is the case $r<m$ and $r<k$, we 
cannot apply directly the area formula \cite[Theorem 3.2.3]{Fe} or the coarea
formula \cite[Theorem 3.2.11]{Fe}.

\begin{proposition}\label{dotbound} Let $\Lambda\subset\mathbb R^k$ be an open 
set and let $\{\Pi_{V_\lambda}\mid\lambda\in\Lambda\}$ be a 
family of orthogonal projections in $\mathbb R^n$ onto $m$-planes. Suppose 
that the mapping $\lambda\mapsto V_\lambda$ has a uniformly continuous 
derivative and there exists $C_0>0$ with $\Vert D_\lambda V_\lambda\Vert<C_0$
for all $\lambda\in\Lambda$. Fix $\lambda^0\in\Lambda$. Assume that there are
$r\le m$ and $d>0$ such that for any $z\in V_{\lambda^0}^\perp$
\begin{equation}\label{bigvol}
\Vert\wedge_r D_\lambda\Pi_{V_{\lambda^0}}(z)\Vert>d|z|^r.
\end{equation}
Then there exist $C>0$ and $R>0$ such that for all $\delta>0$ and
for all $x\ne y\in\mathbb R^n$ we have  
\[
\mathcal L^k(\{\lambda\in B(\lambda^0,R)\mid |\Pi_{V_\lambda}(x-y)|\le\delta\})
  \le C\delta^r|x-y|^{-r}.
\]
\end{proposition}

\begin{proof} We may restrict our consideration to the case 
$|x-y|=1$ and $0<\delta<\delta_0$ for some 
$0<\delta_0<\frac 12$. Let $z\in\mathbb R^n$ be such that $|z|=1$ and 
$|\Pi_{V_{\lambda^0}}(z)|<\delta$. Writing
$z=z_1+z_2$ where $z_1=\Pi_{V_{\lambda^0}}(z)\in V_{\lambda^0}$ and 
$z_2\in V_{\lambda_0}^\perp$, we have $|z_2|>\frac 12$
since $|z_1|<\delta<\frac 12$. By 
\eqref{projbound} and \eqref{bigvol}
\[
\Vert\wedge_r D_\lambda\Pi_{V_{\lambda^0}}(z)\Vert
  \ge\Vert\wedge_r D_\lambda\Pi_{V_{\lambda^0}}(z_2)\Vert>2^{-r}d,
\]
implying the existence of an $r$-dimensional subspace $U\subset\mathbb R^k$ 
such that the restriction of $D_\lambda\Pi_{V_{\lambda^0}}(z)$ to $U$ is 
injective and
$|\det(D_\lambda\Pi_{V_{\lambda^0}}(z)\vert_U)|>2^{-r}d$.
Since the mapping $\lambda\mapsto V_\lambda$ has uniformly continuous 
derivative and the mapping $V\mapsto\Pi_V(z)$ is smooth there exists $R'>0$ 
such that the restriction of $D_\lambda\Pi_{V_{\lambda^1}}(z)$ to $U$ is 
injective (with the same lower bound for the derivative as above) 
for all $\lambda^1\in(\lambda^0+U^\perp)\cap B(\lambda^0,R')$. We denote
by $T^{\lambda^1}$ the restriction of the mapping 
$\lambda\mapsto\Pi_{V_\lambda}(z)$ to $(\lambda^1+U)\cap\Lambda$.

By the above arguments, $C_1<|\det DT^{\lambda^1}|<C_2$ for some constants 
$C_1>0$ and $C_2>0$. Combining this with $\Vert D_\lambda V_\lambda\Vert<C_0$,
we obtain that the singular values of $DT^{\lambda^1}$ are uniformly bounded 
from above and below. Therefore,   
a quantitative version of the inverse function theorem \cite[Lemma 3.1]{JJLL} 
gives that there exist $a>0$ and $R>0$ such that for all 
$\lambda^1\in(\lambda^0+U^\perp)\cap B(\lambda^0,\tfrac{R'}2)$ and for all
$\lambda\in(\lambda^1+U)\cap B(\lambda^1,\tfrac{R'}2)$  the mapping 
$T^{\lambda^1}$ is a diffeomorphism onto its image in $B(\lambda,R)$, and 
moreover, the inclusion 
\begin{equation}\label{lipschitz}
B(T^{\lambda^1}(\lambda),a\rho)\cap T^{\lambda^1}(B(\lambda,3R))\subset 
  T^{\lambda^1}(B(\lambda,\rho))
\end{equation}
is valid for all $0<\rho<3R$. For each 
$\lambda^1\in(\lambda^0+U^\perp)\cap B(\lambda^0,R)$, let $\hat\lambda$ be a 
minimum point of $|T^{\lambda^1}|$ in $B(\lambda^1,R)$. We may assume that 
$|T^{\lambda^1}(\hat\lambda)|\le\delta$. By \eqref{lipschitz} we have 
$|T^{\lambda^1}(\lambda)|>\delta$ for any 
$\lambda\in\big((\lambda^1+U)\cap B(\lambda^1,R)\big)\setminus 
  B(\hat\lambda,2\delta a^{-1})$. 
Thus for any $\lambda^1\in(\lambda^0+U^\perp)\cap B(\lambda^0,R)$ we obtain
\[
\mathcal L^r(\{\lambda\in (\lambda^1+U)\cap B(\lambda^0,R)\mid 
  |\Pi_{V_\lambda}(z)|\le\delta\})\le\mathcal L^r(B(0,1))\left(\frac 2a\right)^r
    \delta^r.
\]
The claim follows by Fubini's theorem.
\end{proof}  

In the following lemma we compare projections onto $m$-planes to those 
onto certain
extended $(m+p)$-planes.

\begin{lemma}\label{extendedder}
Let $a,b\in\mathbb R$ and let $V_\cdot:(a,b)\to G(n,m)$ be continuously
differentiable. Assume that $c\in (a,b)$, $p$ is an integer with $0<p<n-m$ 
and $U\subset V_c^\perp$ is a $p$-plane. Then there exist $a',b'\in (a,b)$ 
such that 
$c\in (a',b')$ and the function $\widetilde V_\cdot:(a',b')\to G(n,m+p)$ 
defined by 
$\widetilde V_s=\langle V_s,U\rangle$ for $s\in (a',b')$ is well-defined, 
continuously differentiable and 
\[
\left.\frac{\partial\Pi_{V_s}(z)}{\partial s}\right|_{s=c}
   =\left.\frac{\partial\Pi_{\widetilde V_s}(z)}{\partial s}\right|_{s=c}
\]
for all $z\in\langle V_c,U\rangle^\perp$. 
\end{lemma}

\begin{proof}
By the continuity of $V_\cdot$, there exists a neighbourhood $(a',b')\ni c$
such that $V_s\cap U=\{0\}$ for all $s\in (a',b')$ implying that 
$\widetilde V_s\in G(n,m+p)$ is well-defined. Clearly, $\widetilde V_\cdot$ is
continuously differentiable. Note that for 
$z\in\langle V_c,U\rangle^\perp$ we have
$\Pi_{\widetilde V_s}(z)\in U^\perp\cap\widetilde V_s$. Furthermore, 
$U^\perp\cap\widetilde V_s$ is an $m$-dimensional plane having distance of 
order $\mathcal O(s-c)$ to 
$V_s$ since the distances of $V_c$ to $V_s$ and
to $U^\perp\cap\widetilde V_s$ are of order $\mathcal O(s-c)$. Combining this 
with the fact $\Pi_{V_s}(z)=\Pi_{V_s}(\Pi_{\widetilde V_s}(z))$, we conclude that
the angle $\theta_s$ between $\Pi_{V_s}(z)$ and $\Pi_{\widetilde V_s}(z)$ is of 
order $\mathcal O(s-c)$. Clearly, both $|\Pi_{V_s}(z)|$ and 
$|\Pi_{\widetilde V_s}(z)|$ are of order $\mathcal O(s-c)$, and therefore, 
$|\Pi_{V_s}(z)-\Pi_{\widetilde V_s}(z)|=|\sin\theta_s\Pi_{\widetilde V_s}(z)|$
is of order $\mathcal O ((s-c)^2)$ giving the claim.
\end{proof}

Now we are ready to prove Theorem~\ref{projresult}. We will apply 
Lemma~\ref{lemma1} to a parametrized family of projections onto $(m+p)$-planes 
for a suitable $p$. The role of Proposition~\ref{dotbound} is to imply
that the assumptions of Lemma~\ref{lemma1} are valid. However, the 
dimension $k$ of the parameter space is too small to guarantee the validity 
of the assumptions of 
Proposition~\ref{dotbound}. To overcome this problem we extend the 
parameter space, and for this purpose, we need the local coordinates defined in
Remark~\ref{localcoor}.

\begin{proof}[Proof of Theorem~\ref{projresult}]
Let $l$ be an integer with $0\le l\le m-1$.  By 
\eqref{naturalbounds} we may assume that $p(l)<n-m$. Writing 
$p=p(l)$, it follows from \eqref{p} (see also Remark~\ref{full} c) and 
Figure~\ref{pic5}) that
\begin{equation}\label{klimits}
l(n-m)+(n-m-p-1)(m-l)<k\le l(n-m)+(n-m-p)(m-l).
\end{equation}
Consider $\lambda^0\in\Lambda$.
It is clearly enough to prove the claim in $B(\lambda^0,R)$ for some $R>0$
such that $B(\lambda^0,R)\subset\Lambda$. 
Let $R'>0$ be such that $\overline B(\lambda^0,R')\subset\Lambda$, where 
$\overline B(\lambda^0,R')$ is the closure of $ B(\lambda^0,R')$. 
By Remark~\ref{vgenerates} the formula
\[
A_i(z)=\left.\frac{\partial \Pi_{V_\lambda}(z)}{\partial\lambda_i}
  \right|_{\lambda=\lambda^0}=D_VF(V_{\lambda^0},z)\circ D_{\lambda}V_\lambda(u_i)
\]
defines a linear map $A_i:V_{\lambda^0}^\perp\to V_{\lambda^0}$ for all 
$i=1,\dots,k$, where $\{u_1,\dots,u_k\}$ is the natural basis of $\mathbb R^k$.
Letting $t=n-m-p$, we have by \eqref{klimits} that $k>m(t-1)+l(n-m-t+1)$. 
Since $D_\lambda V_\lambda$ is injective and 
$D_VF(V_{\lambda^0},\cdot)$ is bijective (see Remark~\ref{vgenerates}), the
assumptions of Lemma~\ref{independent} are valid. 
We denote by
$\{\hat e_{m+1},\dots,\hat e_{m+t}\}$ an orthonormal basis of the space 
$W\subset V_{\lambda^0}^\perp$ given by Lemma~\ref{independent} and by
$\{\hat e_{m+1},\dots,\hat e_{m+t},\hat e_{m+t+1},\dots,\hat e_n\}$ the extension
of it to an orthonormal basis of $V_{\lambda^0}^\perp$. Letting 
$\delta<\frac 12$, define an extended parameter space 
$\widetilde\Lambda=B(\lambda^0,R')\times\prod_{i=m+t+1}^n\prod_{j=m+1}^{m+t} 
  ]-\delta,\delta[$
and write $\tilde\lambda=(\tilde\lambda^1,\tilde\lambda^2)\in\widetilde\Lambda$
where $\tilde\lambda^1\in B(\lambda^0,R')$ and 
$\tilde\lambda^2\in]-\delta,\delta[^{pt}$. 

For each $\tilde\lambda\in\widetilde\Lambda$, define an $(m+p)$-dimensional 
plane 
\[
\widetilde V_{\tilde\lambda}=\langle V_{\tilde\lambda^1},
  \hat e_{m+t+1}(\tilde\lambda^2),\dots,\hat e_n(\tilde\lambda^2)\rangle
\]
where 
$\hat e_i(\tilde\lambda^2)=\prod_{j=m+1}^{m+t}R^{ij}((\tilde\lambda^2)_{ij})
  \hat e_i$ 
is as in Remark~\ref{localcoor} for all $i=m+t+1,\dots,n$. Decreasing $R'$ if
necessary guarantees that the plane $\widetilde V_{\tilde\lambda}$ is 
$(m+p)$-dimensional for every $\tilde\lambda\in\widetilde\Lambda$. 
In this way we obtain a $\tilde k$-dimensional family 
$\widetilde V_{\tilde\lambda}$ of $(m+p)$-planes, where $\tilde k=k+pt$. For 
$i=1,\dots,\tilde k$, the indices
$1,\dots,k$ correspond to $\tilde\lambda^1$ and the remaining indices 
$k+1,\dots,\tilde k$ correspond to $\tilde\lambda^2$.

Note that $\widetilde V_{\tilde\lambda^0}^\perp=W$ where 
$\tilde\lambda^0=(\lambda^0,0)$. Since 
$\hat e_j(\tilde\lambda)$ is independent of $\tilde\lambda^1$ for 
$j=m+t+1,\dots,n$ we conclude from Lemma~\ref{extendedder} that
\begin{equation}\label{part1}
\left.\frac{\partial\Pi_{\widetilde V_{\tilde\lambda}}(z)}{\partial\tilde\lambda_i}
  \right|_{\tilde\lambda=\tilde\lambda^0}
  =\left.\frac{\partial \Pi_{V_\lambda}(z)}{\partial\lambda_i}
  \right|_{\lambda=\lambda^0}=A_i(z)\in V_{\lambda^0}\subset
  \widetilde V_{\tilde\lambda^0}
\end{equation}
for all $z\in W$ and for all $i\in\{1,\dots,k\}$. 
Fix $z\in W$ and let $j_1,\dots,j_{l+1}\in\{1,\dots,k\}$ be the indices for
which \eqref{bigvol2} is satisfied with $d'>0$. Let 
$|z_{j_0}|=\max_j\{|z_j|\}$. 
Denote by $j_{l+2},\dots,j_{l+1+p}\in\{k+1,\dots,\tilde k\}$ the indices
determined by $(\tilde\lambda^2)_{ij_0}$, where $i=m+t+1,\dots,n$. By 
the definition of the extension and Remark~\ref{localcoor},
we obtain for $h=l+2,\dots,l+1+p$
\begin{equation}\label{part2}
\left.\frac{\partial\Pi_{\widetilde V_{\tilde\lambda}}(z)}
   {\partial\tilde\lambda_{j_h}}\right|_{\tilde\lambda=\tilde\lambda^0}
   =z_{j_0}\hat e_i\in V_{\lambda^0}^\perp\cap\widetilde V_{\tilde\lambda^0},
\end{equation}
where $i$ is determined by $h$. Let 
$\xi=u_{j_1}\wedge\dots\wedge u_{j_{l+1+p}}$, where $\{u_1,\dots,u_{\tilde k}\}$
is the natural basis of $\mathbb R^{\tilde k}$. Now \eqref{part1}, \eqref{part2},
\eqref{perpnorm} and the fact that
$|z_{j_0}|\ge\frac{|z|}{\sqrt t}$ combine to give for $r=l+1+p$ that
\begin{align*}
&\Vert\wedge_r D_{\tilde\lambda}\Pi_{\widetilde V_{\tilde\lambda^0}}(z)\Vert
  \ge\Vert\wedge_r D_{\tilde\lambda}\Pi_{\widetilde V_{\tilde\lambda^0}}(z)(\xi)
  \Vert\\
&=\Vert A_{j_1}(z)\wedge\dots\wedge A_{j_{l+1}}(z)\Vert\cdot
  \Vert z_{j_0}\hat e_{m+t+1}\wedge\dots\wedge z_{j_0}\hat e_n\Vert
  >\frac{d'}{(\sqrt t)^p}|z|^{l+1+p}.
\end{align*}
Hence, the assumptions of Proposition~\ref{dotbound} are valid for the
extended family 
$\{\Pi_{\widetilde V_{\tilde\lambda}}\mid\tilde\lambda\in\widetilde\Lambda\}$
(the bounds are uniform since we consider only the compact 
set $\overline B(\lambda^0,R')$).

Applying Proposition~\ref{dotbound} to the family 
$\{\Pi_{\widetilde V_{\tilde\lambda}}\mid\tilde\lambda\in\widetilde\Lambda\}$
implies that the assumptions of Lemma~\ref{lemma1} are valid for the family
$\{\Pi_{\widetilde V_{\tilde\lambda}}\mid\tilde\lambda\in B(\tilde\lambda^0,R)\cap
  \widetilde\Lambda\}$,
where $R$ is as in Proposition~\ref{dotbound}.
Under the assumption $\dim\mu\le r$ Lemma~\ref{lemma1} gives
$\dim(\Pi_{\widetilde V_{\tilde\lambda}})_\ast\mu=\dim\mu$
for 
$\mathcal L^{\tilde k}$-almost all 
$\tilde\lambda\in B(\tilde\lambda^0,R)\cap\widetilde\Lambda$. 
Moreover, from \eqref{naturalbounds} we deduce that 
\begin{equation}\label{added}
\dim(\Pi_{V_{\tilde\lambda^1}}\circ\Pi_{\widetilde V_{\tilde\lambda}})_\ast\mu
\ge\dim(\Pi_{\widetilde V_{\tilde\lambda}})_\ast\mu-p
\end{equation}
for every $\tilde\lambda$.
Observing that
\begin{equation}\label{added2}
\Pi_{V_{\tilde\lambda^1}}=\Pi_{V_{\tilde\lambda^1}}\circ
  \Pi_{\widetilde V_{\tilde\lambda}},
\end{equation}
the first inequality in \eqref{lowerlimit} follows from Fubini's theorem. The 
second inequality in \eqref{lowerlimit} can be verified similarly:  By 
Lemma~\ref{lemma1} we have $\dim(\Pi_{\widetilde V_{\tilde\lambda}})_\ast\mu\ge r$
for $\mathcal L^{\tilde k}$-almost all 
$\tilde\lambda\in B(\tilde\lambda^0,R)\cap\widetilde\Lambda$ provided that
$\dim\mu>r$. As before, \eqref{added}, \eqref{added2} and Fubini's theorem
combine to give the second inequality in \eqref{lowerlimit}.
Finally, assuming that $\dim\mu>p(m-1)+m$, we get from Lemma~\ref{lemma1}
that for $\mathcal L^{\tilde k}$-almost all 
$\tilde\lambda\in B(\tilde\lambda^0,R)\cap\widetilde\Lambda$ the projected
measure $(\Pi_{\widetilde V_{\tilde\lambda}})_\ast\mu$ is absolutely continuous 
with respect to $\mathcal H^{p(m-1)+m}$, and therefore, 
$(\Pi_{V_{\tilde\lambda^1}}\circ\Pi_{\widetilde V_{\tilde\lambda}})_\ast\mu$ is 
absolutely continuous with respect to $\mathcal H^m$. Again, the claim follows 
from \eqref{added} and Fubini's theorem.

It remains to prove that the lower bounds and the condition for the absolute
continuity are the best possible ones.
Let $l$, $p$ and $k$ be as in \eqref{klimits} and fix $0\le s\le 1$. We start
by constructing a $k$-dimensional family
$\{\Pi_{V_\lambda}\mid\lambda\in\Lambda\}$ of projections and a measure $\mu$ on 
$\mathbb R^n$ with $\dim\mu=p+l+s$ such that 
\[
\dim(\Pi_{V_\lambda})_*\mu=\dim\mu-p
\] 
for $\mathcal L^k$-almost all $\lambda\in\Lambda$. 
Let $\Lambda=]-\delta,\delta[^k$
and consider a family $\{\Pi_{V_\lambda}\mid\lambda\in\Lambda\}$  
constructed similarly as the above extension using the rotations illustrated in 
Figure~\ref{pic5}. Define
$\mu=\nu_1\times\nu_2$ where $\nu_1$ is a $s$-dimensional measure on the 
space spanned by $e_{l+1}$ and $\nu_2$ is the restriction of 
$\mathcal L^{l+p}$ to the unit ball of the space  
$X=\langle e_1,\dots,e_l,e_{n-p+1},\dots,e_n\rangle$. Then $\dim\mu=l+p+s$. 
Since $e_j(\lambda)\in X^\perp$ for all $\lambda\in\Lambda$ and for all 
$j=l+1,\dots,m$, we have $(\Pi_{V_\lambda})_*\nu_2=(\Pi_{W_\lambda})_*\nu_2$ where
$W_\lambda=\langle e_1(\lambda),\dots,e_l(\lambda)\rangle$, and therefore,
$\dim(\Pi_{V_\lambda})_*\nu_2\le l$ for all $\lambda\in\Lambda$. The fact that
$\dim(\Pi_{V_\lambda})_*\nu_1\le s$ gives
$\dim(\Pi_{V_\lambda})_\ast\mu\le l+s$ for all $\lambda\in\Lambda$ implying the 
sharpness of the first inequality in \eqref{lowerlimit}. 

The sharpness of the
second inequality in \eqref{lowerlimit} is verified similarly
by letting $\mu$ to
be any measure on $X$ with $p(l-1)+l\le\dimH\mu\le p(l)+l$. Finally, if
$l=m-1$, define $\mu=\nu_1\times\nu_2$ where $\nu_1$ is 
the 1-dimensional Hausdorff measure restricted to the 1-dimensional four corner
Cantor set in $\langle e_m,e_{m+1}\rangle$ and $\nu_2$ is as above. Then for 
$\mathcal L^k$-almost all $\lambda\in\Lambda$ the projection of $\mu$ to
$\langle e_m(\lambda)\rangle$ is singular with respect to $\mathcal H^1$ on 
$\langle e_m(\lambda)\rangle$, implying the singularity of 
$(\Pi_{V_\lambda})_*\mu$ with respect to $\mathcal H^m$ on $V_\lambda$.   
\end{proof}

\begin{remark}\label{bestone}
a) The lower bounds given in Theorem~\ref{projresult} are the 
best possible ones in the sense that for each $d$ there exist a measure $\mu$ 
with $\dim\mu=d$ and a family of projections such that the lower bounds are 
achieved. However, this does not mean that for any family and any $d$ one 
could construct such a measure. Different families have different
lower bounds - even in the case $k=1$, see \cite[Remark 3.5]{JJLL}.

b) In the setting of \cite{F2} the study of non-existence of
Besicovitch $(n,k)$-sets leads to
a $k$-dimen\-sional family of projections from $\mathbb R^{(k+1)(n-k)}$ onto 
$\mathbb R^{n-k}$. The set (or the measure) one is projecting is 
$k(n-k)$-dimensional. The essential step is to show that  
projections have positive measure for almost all parameters.
In the case of Besicovitch $(n,n-1)$-sets $k=n-1$, which leads to an 
$(n-1)$-dimensional family of projections from $\mathbb R^n$ onto 
$\mathbb R^1$. The set one is projecting is $(n-1)$-dimensional. According to 
Theorem~\ref{projresult}, projections have positive measure
provided that $n-1>1$ implying that there are no Besicovitch 
$(n,n-1)$-sets for $n\ge 3$. For other values of $k$ the dimension of the
parameter space is too small in order to apply Theorem~\ref{projresult}. 

There exist valid proofs for the non-existence of 
Besicovitch $(n,k)$-sets for $k>\frac n2$ by Falconer \cite{F1} and for
$2^{k-1}+k\ge n\ge 3$ by Bourgain \cite{B} (see also \cite{O}). However, the 
method of \cite{F2} would be more elementary in the sense that it does not
use Fourier transform.  
\end{remark}

\section{Families of smooth maps}\label{generalfamilies}

In this section we discuss the extension of Theorem~\ref{projresult} 
to families of smooth maps from $\mathbb R^n$ to $\mathbb R^m$.
Note that an orthogonal projection is uniquely determined by its kernel,
and moreover, the restriction of a linear map to the orthogonal complement of
its kernel is a diffeomorphism onto its image. These 
simple observations lead to the following definition.

\begin{definition}\label{defgeneralfull}
Let $\Lambda\subset\mathbb R^k$ be open and let 
$\mathcal F=\{F_\lambda:\mathbb R^n\to\mathbb R^m\mid\lambda\in\Lambda\}$ 
be a family of $C^2$-maps. 
Define $V_\lambda^x:=\ker(D_xF_\lambda(x))^\perp$. The family $\mathcal F$ is 
\emph{non-degenerate} if the following conditions are satisfied:
\begin{enumerate}
\item The plane $V_\lambda^x$ is $m$-dimensional for all $x\in\mathbb R^n$ and
$\lambda\in\Lambda$ and the family 
$\{\Pi_{V_\lambda^x}\mid\lambda\in\Lambda\}$ is non-degenerate for all 
$x\in\mathbb R^n$.
\item The map $x\mapsto D_\lambda\Pi_{V_\lambda^x}$ is continuous.
\item There exist constants $C_1,C_2>0$ such that 
$\Vert D_\lambda D_xF_\lambda(x)\Vert\le C_1$ and 
$\Vert D_\lambda D_x^2F_\lambda(x)\Vert\le C_2$ for all $x\in\mathbb R^n$ and
$\lambda\in\Lambda$. 
\end{enumerate}
Here the derivatives with respect to $x$ and $\lambda$ are denoted by $D_x$ and
$D_\lambda$, respectively, and the norm of a linear map is denoted by
$\Vert\cdot\Vert$. 
\end{definition}

\begin{remark}\label{genfull}
Restricting our consideration to compact sets $K_1\subset\Lambda$ and 
$K_2\subset\mathbb R^n$, we may assume that the constants in 
Definition~\ref{defgeneralfull} are independent of $x$, the map 
$x\mapsto D_\lambda\Pi_{V_\lambda^x}$ is uniformly continuous and there exists 
a constant $d>0$ such that $|\det D_xF_\lambda(x)\vert_{V_\lambda^x}|>d$.
Condition (3) is valid if $D_\lambda D_xF_\cdot(\cdot)$ and
$D_\lambda D_x^2F_\cdot(\cdot)$ are assumed to be continuous.
\end{remark}

\begin{theorem}\label{generalresult}
Let $\Lambda\subset\mathbb R^k$ be an open set and let 
$\mu$ be a finite Radon measure on $\mathbb R^n$ with compact support.
Assume that the family 
$\mathcal F=\{F_\lambda:\mathbb R^n\to\mathbb R^m\mid\lambda\in\Lambda\}$ of 
$C^2$-maps is non-degenerate. Then for all $l=0,\dots,m-1$ and 
for $\mathcal L^k$-almost all $\lambda\in\Lambda$ 
\begin{equation}\label{lowerlimit2}
\dim (F_\lambda)_*\mu\ge
\begin{cases} \dim\mu-p(l), &\text{if }p(l)+l\le\dim\mu\le p(l)+l+1,\\
               l+1, &\text{if }p(l)+l+1\le\dim\mu\le p(l+1)+l+1,
\end{cases}
\end{equation}
where $p(l)$ is as in \eqref{p}. Furthermore, for $\mathcal L^k$-almost all 
$\lambda\in\Lambda$ the image measure $(F_\lambda)_*\mu$ is absolutely 
continuous with respect $\mathcal H^m$ provided that $\dim\mu>p(m-1)+m$.
The lower bounds given in \eqref{lowerlimit2} and the condition for the 
absolute continuity are the best possible ones.
\end{theorem}

\begin{proof} We proceed as in the proof of Theorem~\ref{projresult}.
The essential step is to define an extended family 
$\widetilde F_{\tilde\lambda}:\mathbb R^n\to\mathbb R^{m+p}$ for which
the assumptions of Lemma~\ref{lemma1} are valid. 

Let $x^0\in\mathbb R^n$ and $\lambda^0\in\Lambda$. Let $R'>0$ and 
$\varepsilon>0$ be sufficiently small.
We identify the range of $F_\lambda$ with 
$V_\lambda^{x^0}=\ker(D_xF_\lambda(x^0))^\perp$ such that 
$F_\lambda(x^0)=0\in V_\lambda^{x^0}$. As in the 
proof of Theorem~\ref{projresult}, the $k$-dimensional family of 
$m$-planes $\{V_\lambda^{x^0}\mid\lambda\in B(\lambda^0,R')\}$ is extended to a 
$\tilde k$-dimensional family of $(m+p)$-planes 
$\{\widetilde V_{\tilde\lambda}\mid\tilde\lambda\in\widetilde\Lambda\}$.
Denoting the orthogonal complement of $V_\lambda^{x_0}$ inside  
$\widetilde V_{\tilde\lambda}$ by $V_{\tilde\lambda}^N$ and identifying 
$\mathbb R^{m+p}$, $V_\lambda^{x^0}\times V_{\tilde\lambda}^N$ and
$\widetilde V_{\tilde\lambda}$ with each other, 
the extended family 
$\widetilde{\mathcal F}=\{\widetilde F_{\tilde\lambda}:\mathbb R^n
  \to\mathbb R^{m+p}\mid\tilde\lambda\in\widetilde\Lambda\}$ 
is defined by 
$\widetilde F_{\tilde\lambda}(y)=(F_\lambda(y),
  \Pi_{V_{\tilde\lambda}^N}(y))$.

It is enough to prove that the assumption \eqref{smallmeasure} is satisfied
for all 
$x,y\in B(x^0,\varepsilon)$ such that 
$|\widetilde F_{\tilde\lambda}(y)-\widetilde F_{\tilde\lambda}(x)|\le
  \tilde\delta|y-x|$ 
for some small $\tilde\delta$. Writing
$G(\tilde\lambda)=\widetilde F_{\tilde\lambda}(y)
    -\widetilde F_{\tilde\lambda}(x)$
and observing that $D_xF_\lambda(x)(v)=D_xF_\lambda(x)(\Pi_{V_\lambda^x}(v))$ for
any $v\in\mathbb R^n$, we have by Taylor's formula
\begin{align*}
D_{\tilde\lambda}G(\tilde\lambda)&=\Bigl(D_{\tilde\lambda}\bigl(D_xF_\lambda(x)
 (\Pi_{V_\lambda^x}(y-x))+\frac 12 D_x^2F_\lambda(\xi)(y-x)\bigr),D_{\tilde\lambda}
   \Pi_{V_{\tilde\lambda}^N}(y-x)\Bigr)\\
  &=\Bigl(D_xF_\lambda(x)\bigl(D_{\tilde\lambda}\Pi_{V_\lambda^{x^0}}(y-x)
   +D_{\tilde\lambda}\Pi_{V_\lambda^x}(y-x)
   -D_{\tilde\lambda}\Pi_{V_\lambda^{x^0}}(y-x)\bigr)\\
  &+D_{\tilde\lambda}D_xF_\lambda(x)(\Pi_{V_\lambda^x}(y-x))+\frac 12 
    D_{\tilde\lambda}D_x^2F_\lambda(\xi)(y-x),D_{\tilde\lambda}
    \Pi_{V_{\tilde\lambda}^N}(y-x)\Bigr)\\
  &=(D_xF_\lambda(x)\oplus\Id)(D_{\tilde\lambda}
    \Pi_{\widetilde V_{\tilde\lambda}}(y-x))\\
  &+\Bigl(D_xF_\lambda(x)\bigl(D_{\tilde\lambda}\Pi_{V_\lambda^x}(y-x)
   -D_{\tilde\lambda}\Pi_{V_\lambda^{x^0}}(y-x)\bigr)\\
  &+D_{\tilde\lambda}D_xF_\lambda(x)(\Pi_{V_\lambda^x}(y-x))
   +\frac 12 D_{\tilde\lambda}D_x^2F_\lambda(\xi)(y-x),0\Bigr).
\end{align*} 
Note that in the above sum the norm of the second term does not change 
when replacing 
$D_{\tilde\lambda}$ by $D_\lambda$. By Remark~\ref{genfull}, we have
$|\Pi_{V_\lambda^x}(y-x)|\le\tilde d\tilde\delta|y-x|$ for some $\tilde d$,
and therefore, by Definition~\ref{defgeneralfull},
the norm of the second term is less than $\tilde\varepsilon|y-x|$ where 
$\tilde\varepsilon=\tilde\varepsilon(\varepsilon,\tilde\delta)$ tends to zero 
as $\varepsilon$ and $\tilde\delta$ tend to zero. This, in turn, implies that 
$D_{\tilde\lambda}G(\tilde\lambda)$ is a small 
perturbation of a diffeomorphic image of  
$D_{\tilde\lambda}\Pi_{\widetilde V_{\tilde\lambda}}(y-x)$. According to the proof 
of Proposition~\ref{dotbound}, the singular values of 
$D_{\tilde\lambda}\Pi_{\widetilde V_{\tilde\lambda}}(y-x)$
are bounded from above and below when $\tilde\lambda$ is restricted to a 
suitable subspace. (The restriction is denoted by $T^{\lambda^1}$ in the proof 
of Proposition~\ref{dotbound}.) Thus the same is true for  
$D_{\tilde\lambda}G(\tilde\lambda)$, and from \cite[Lemma 3.1]{JJLL} we conclude
that a 
suitable restriction of $G$ is a diffeomorphism with uniform lower and upper 
bounds. Proceeding as in the 
proof of Proposition~\ref{dotbound}, we have for all $\delta>0$ and 
for all $x\ne y\in B(x^0,\varepsilon)$ 
\[
\mathcal L^k(\{\lambda\in B(\lambda^0,R)\mid |\widetilde F_{\tilde\lambda}(y)
  -\widetilde F_{\tilde\lambda}(x)|\le\delta\})
  \le C\delta^r|y-x|^{-r}.
\] 
The rest of the proof follows similarly as that of Theorem~\ref{projresult}.
\end{proof}

\begin{remark}\label{PSremark}
It is natural to consider whether the part of \cite[Theorem 7.3]{PS} concerning
the exceptional sets of projections is useful in our setting. For this purpose,
one needs to extend a $k$-dimensional family 
$\{F_\lambda:\mathbb R^n\to\mathbb R^m\mid\lambda\in\Lambda\}$ to a transversal
family 
$\{\widetilde F_{\tilde\lambda}:\mathbb R^n\to\mathbb R^m\mid\tilde\lambda\in
\widetilde\Lambda\}$
for which \cite[Theorem 7.3]{PS} may be applied.
Usually the extended parameter space $\widetilde\Lambda$ is 
$m(n-m)$-dimensional. For the extended family \cite[(7.4)]{PS} reads in our 
notation as follows
\begin{equation}\label{PSineq}
\dim\{\tilde\lambda\in\widetilde\Lambda\mid\dim(\widetilde F_{\tilde\lambda})_*
\mu\le\sigma\}\le m(n-m)+\sigma-\alpha
\end{equation}
where $I_\alpha(\mu)<\infty$. Inequality \eqref{PSineq} gives a lower bound
for $\mathcal L^k$-almost all $\lambda\in\Lambda$ provided that
$m(n-m)+\sigma-\alpha<k$. Recalling \eqref{energydim}, the best possible lower
bound obtained in this way is
\begin{equation}\label{PSbound}
\dim(F_\lambda)_*\mu\ge\dim\mu-(m(n-m)-k).
\end{equation}
The lower bound given by Theorem~\ref{generalresult} is better than 
\eqref{PSbound} except in the case where $k\ge (m-1)(n-m)$ and 
$\dim\mu\ge p(m-1)+m-1$. In this case \eqref{PSbound} equals the bound given
by Theorem~\ref{generalresult} but we assume less regularity from the family
than \cite[Theorem 7.3]{PS}. Similarly, our result gives a better bound than 
\cite[(7.6)]{PS} unless $k\ge l+m(n-m)-m$, which implies $p(l)=0$. 
\end{remark}

\end{document}